\documentclass{amsproc}
\usepackage{hyperref}
\usepackage{import}
\usepackage{pdfpages}
\usepackage{transparent}
\usepackage{xcolor}
\usepackage{subcaption}
\usepackage{tikz}
\usepackage{mathtools}
\usepackage{tensor}

\providecommand{\keywords}[1]{\textbf{\textit{Keywords: }} #1}

\usepackage{todonotes}
\usepackage{csquotes}
\usepackage{orcidlink}

\usepackage{bm}

\newtheorem{Theorem}{Theorem}[section]
\newtheorem{Proposition}[Theorem]{Proposition}

\newtheorem{Lemma}[Theorem]{Lemma}

\theoremstyle{definition}
\newtheorem{Definition}[Theorem]{Definition}

\theoremstyle{remark}
\newtheorem{Remark}[Theorem]{Remark}

\usepackage{setspace}
\begin{document}

\title[Singular canards for critical sets with self-intersections]{A topological perspective on singular canards for critical sets with transverse intersections}

\author{Riccardo Bonetto\orcidlink{0000-0001-8075-6147}}
\address{University of Groningen — Bernoulli Institute for Mathematics, Computer Science and Artificial Intelligence; Nijenborgh 9, 9747AG, Groningen, The Netherlands}
\email{r.bonetto@rug.nl}

\author{Hildeberto Jardón-Kojakhmetov\orcidlink{0000-0001-8708-7409}}
\address{University of Groningen — Bernoulli Institute for Mathematics, Computer Science and Artificial Intelligence; Nijenborgh 9, 9747AG, Groningen, The Netherlands}
\email{h.jardon.kojakhmetov@rug.nl}

\subjclass[2020]{...}
\date{...}


\begin{abstract}
   This paper gives a new perspective on singular canards, which is topological in flavor. One key feature is that our construction does not rely on coordinates; consequently, the conditions for the existence of singular canards that we provide are purely geometric. The singularities we study originate at the self-intersection of curves of equilibria of the unperturbed system. Our contribution even allows us to consider degenerate cases of multiple pairwise transverse intersecting  branches of the critical set. We employ stratification theory and algebraic geometric properties to provide sufficient conditions leading to the presence of singular canards. By means of two examples, we corroborate our findings using the well-known blow-up technique.
\end{abstract}

\maketitle



\section{Introduction}

The set of zeros, or equilibria, of a vector field is one of the most fundamental objects to study in order to understand the behaviour of integral curves, or solutions. For this reason, one is also interested in the effect of small perturbations of the former vector field on the set of equilibria. Looking back in history, Fenichel contributed to the development of the theory for normally hyperbolic manifolds and their persistence \cite{Fen1972, FENICHEL197953, hirsch2006invariant}. A remarkable result from Fenichel's theory is that hyperbolic equilibria perturb to invariant objects with the same stability properties. We can therefore say that the behaviour near hyperbolic equilibria is well-understood. However, zeros are not always hyperbolic. So, a further aim of \emph{geometric singular perturbation theory} (GSPT) has been to understand the properties of non-hyperbolic equilibria. It is worth mentioning that, if a zero is semi-hyperbolic by a \emph{centre manifold reduction} \cite{carr1981applications, Vanderbauwhede1989} one can reduce the study to a \emph{fully non-hyperbolic} point. A particular class of non-hyperbolic equilibria on which we will focus is the nilpotent one\footnote{An equilibrium point of a vector field is called nilpotent if the linearisation of the vector field at such a point has only real zero eigenvalues.}.

The presence of a nilpotent singularity can induce the exceptional appearance of \emph{canard} solutions, i.e., solutions that unexpectedly follow a repelling set for a non-negligible amount of time (more detailed explanations will be given later in the text.) Canards have a long history as well \cite{Benoit1981, Eckhaus1983, Dumortier1996, SZMOLYAN2001419}, that reach modern times with the application of the \emph{blow-up} (see appendix \ref{app:blowup}). 

Our aim is to give a coordinate-free description, in line with coordinate-free GSPT \cite{wechselberger2020geometric}, of canard solutions crossing singularities arising at the intersection of two, or more, curves of equilibria. We adapt tools and results from algebra and topology to the problem under consideration. A coordinate-free and topological perspective could set the path to further generalisation of the theory to settings that appear to be challenging under current techniques including e.g., high-dimensions, and degenerate problems.

\section{Singular Canards}\label{sec:strat}

We consider a \emph{two-dimensional real polynomial vector field} 
\begin{equation}\label{eq:problem}
    X=X_0 + \epsilon X_1 ,
\end{equation}
where $\epsilon$ is a small parameter, $0<\epsilon \ll 1$. We omit higher order terms of the form $\epsilon^n X_n$, $n>1$, as they do not contribute to the following analysis.
The critical variety is the set defined by $\mathcal{C} \coloneqq \{ p \in \mathbb{R}^2 \ | \ X_0(p)=0 \}$. The vector field $X_0$ has two orthogonal components, let us call them $A$ and $B$. That is, in Cartesian coordinates one would have $X_0=A\frac{\partial}{\partial x}+B\frac{\partial}{\partial y}$. Let $V(A)$, $V(B)$ be the algebraic curves defined by setting to zero the polynomials $A$, $B$; notice that $\mathcal{C} = V(A) \cap V(B)$. Let us consider the factorisation into irreducible (real) polynomials $A= \prod_k A_k$, $B= \prod_l B_l$ such that $V(A)=\bigcup_k V(A_k)$, $V(B)=\bigcup_l V(B_l)$. Since we are interested in critical varieties arising from the above factorisation, we have the following definition.

\begin{Definition}
$A$ and $B$ have a \emph{common component} if $\exists\, k,\,l$ such that $A_k = B_l$ modulo constants.
\end{Definition}

In particular, we will deal with singular perturbations, which in the current setting appear when the critical variety has dimension $1$.

\begin{Definition}[\cite{wechselberger2020geometric}]
   The perturbation problem \eqref{eq:problem} is singular if $\dim(\mathcal{C})=1$.
\end{Definition}

\begin{Proposition}
The perturbation problem defined by $X=X_0 + \epsilon X_1$ is singular only if either of the following hold:
\begin{enumerate}
    \item $A$ and $B$ are nontrivial and have at least one common component, say $A_k$, $B_l$ such that $\dim\left(V(A_k)\right)=\dim\left(V(B_l)\right)=1$,     \item $A$ is trivial and $\dim\left(V(B)\right)=1$ (or vice versa).
\end{enumerate}
\end{Proposition}

\begin{proof}
If $A$ and $B$ are nontrivial and have no \emph{common components} then $V(A) \cap V(B)$ is a finite set of points \cite{fulton2008algebraic}. Then $A$ and $B$ must have a common component. Moreover, to satisfy $\dim(\mathcal{C})=1$ such a component must be one-dimensional, leading to the first condition. If $A$ is trivial and $\dim(V(B))=1$ then the system is in the \emph{standard form} \cite{kuehn2016multiple}.
\end{proof}
\begin{Remark}
A relevant case where $A$ and $B$ can have a common component is when there is a $\mathbb{Z}_2$ symmetry of the vector field $X_0$. Such a case, together with the higher-dimensional counterparts, arises naturally in the study of coupled cell networks \cite{Golubitsky2006}, i.e., networked dynamical systems inheriting the symmetry properties of the graph structure.
\end{Remark}

Let us consider the case of singular perturbations. We assume that $\mathcal{C}$ has a point, $p_s$, with \emph{intersection number} equal to one \cite{fulton2008algebraic, fulton2012intersection}. In two dimensions, this last condition prescribes the typical geometry of (nondegenerate) transcritical and pitchfork singularities (later in section \ref{sec:degeneracy} we consider degenerate cases where at $p_s$ several curves intersect). For the rest of the paper, we restrict ourselves to a neighbourhood of $p_s$ where there are no other intersection points.
Given the assumptions above, the polynomials $A$ and $B$ have two common components $F_1$, $F_2$ such that $V(F_1) \cap V(F_2) = p_s$ and $V(F_1) \cup V(F_2) =\mathcal{C}$. Roughly speaking, $V(F_1)$ and $V(F_2)$ are two algebraic curves intersecting transversely at the point $p_s$. We define the polynomial $F:= F_1 F_2$ such that $V(F)=\mathcal{C}$.
The set of singular points of $V(F)$ is denoted by $\Sigma V(F)$, and it is given by $\{ p_s\}$.

Let us now consider the perturbation term $X_1$, for which we assume that $\|X_1(p_s) \| =O(1)$, in order to avoid the emergence of new equilibria in a neighbourhood of $p_s$ for $0<\epsilon\ll1$. Moreover, we assume that $X_1$ is such that there are no equilibria of the \emph{reduced flow} in a sufficiently large neighbourhood of $p_s$, see appendix \ref{app:projection}. 
Depending on the structure of the perturbation, and on the stability properties of the critical variety, the system can behave differently around $p_s$. Generically, integral curves that at first are close to an attracting branch of $\mathcal{C}$, when crossing (a neighbourhood of) $p_s$ can either follow another attracting branch or be repelled away. However, there is the possibility that a solution follows a repelling branch of $\mathcal{C}$ for a non-negligible amount of time; such solutions are called \emph{canards}. A similar behaviour can be observed for solution transitioning from a repelling branch to an attracting one; those are usually called \emph{faux-canards}. We are not going to distinguish between the two possible situations. In the following, we state the standard definition of \emph{singular canard}.

\begin{Definition}[\cite{wechselberger2020geometric, SZMOLYAN2001419}]\label{def:sing_canard}
    A singular canard is a solution of the reduced problem passing through the singular point with (nonzero) finite speed.
\end{Definition}

However, there is a problem: in our setting, at $p_s$ the tangent space of the critical variety is not defined. In order to obtain a well-defined problem, we introduce some concepts from stratification theory, which we adapt to the context of the paper.

\begin{Definition}[\cite{cushman2015global}]
    A \emph{stratification} of $\mathcal{C}$ is a collection $\mathfrak{X}$ of smooth manifolds such that
    \begin{enumerate}
        \item $\mathfrak{X}$ is a local partition of $\mathcal{C}$.
        \item $\forall \mathcal{S} \in \mathfrak{X}$, the closure of $\mathcal{S}$ in $\mathcal{C}$ is the union of $\mathcal{S}$ and $\{ \mathcal{S}' \in  \mathfrak{X} \ | \ \dim(\mathcal{S}) =0 \}$.
    \end{enumerate}
\end{Definition}

\begin{Definition}[\cite{gibson1976topological}]\label{def:whitney}
A \emph{Whitney stratification} is a stratification $\Tilde{\mathfrak{X}}$ that satisfies the \emph{Whitney regularity condition}.
\end{Definition}

The Whitney regularity condition ensures that along a stratum $\Tilde{\mathcal{S}} \in \Tilde{\mathfrak{X}}$ the topological type of $\Tilde{\mathcal{S}}$ does  not change; in terms of the map $X_0$, the rank of $\text{D} X_0|_{\Tilde{\mathcal{S}}}$ is maximal along each stratum, where $\text{D} X_0$ denotes the Jacobian. 
We say that a stratification is \emph{minimal} if it is refined by any other, unless otherwise stated we always refer to minimal stratifications. For such a reason, the word minimal will be omitted.

We recall that any algebraic set in $\mathbb{R}^n$ admits a Whitney stratification \cite{Hironaka1973}. Considering the filtration $V(F)\supseteq \Sigma V(F)$, the Whitney stratification of the critical variety $\mathcal C$ is given by $\Tilde{\mathfrak{X}} =\{\Tilde{\mathcal{S}}_1, \Tilde{\mathcal{S}}_2, \Tilde{\mathcal{S}}_3, \Tilde{\mathcal{S}}_4, \Tilde{\mathcal{S}}_5 \}$, where, without loss of generality, we can assume that $\dim(\Tilde{\mathcal S}_i)=1$, for $i=1,2,3,4$ and $\dim(\Tilde{\mathcal{S}}_5)=0$, see figure \ref{fig:strat_map}. Clearly, the stratum $\Tilde{\mathcal{S}}_5$ is the point of intersection, $\{p_s\}=\Sigma V(F)$. 

\begin{Definition}[\cite{cushman2015global}]\label{def:stratified_vf}
    A \emph{stratified vector field} on a stratification is a map that assigns to each stratum a smooth vector field.
\end{Definition}

We define a map, $\rho$, acting on vector fields on the Whitney stratification, $\Tilde{\mathfrak{X}}$, such that for the one-dimensional strata it coincides with the projection $\pi$ (see appendix \ref{app:projection}) of a vector field along fast fibres, while for the zero-dimensional stratum, it returns the restriction of the vector field to the point. Let us notice that $\rho \circ X_1$ gives a stratified vector field on the one-dimensional (Whitney) strata.

Next, let us relax the Whitney condition on the stratification of $\mathcal{C}$, and let us consider stratifications $\mathfrak{X}=\{\mathcal{S}_1, \mathcal{S}_2, \mathcal{S}_3 \}$, where one of the strata, say $\mathcal{S}_3$, is now given by the union of two one-dimensional Whitney strata with the zero-dimensional stratum, e.g., $\mathcal{S}_3 = \Tilde{\mathcal{S}}_2 \cup \Tilde{\mathcal{S}}_4 \cup \Tilde{\mathcal{S}}_5$, see figure \ref{fig:strat_map}. Let us remark that, by the definition of stratification, we consider only the combinations that give rise to smooth strata (and are minimal).

\begin{figure}[htbp]
\centering
\begin{tikzpicture}[scale=.8] 
\node at (0,1){
 \includegraphics[scale=0.65]
 {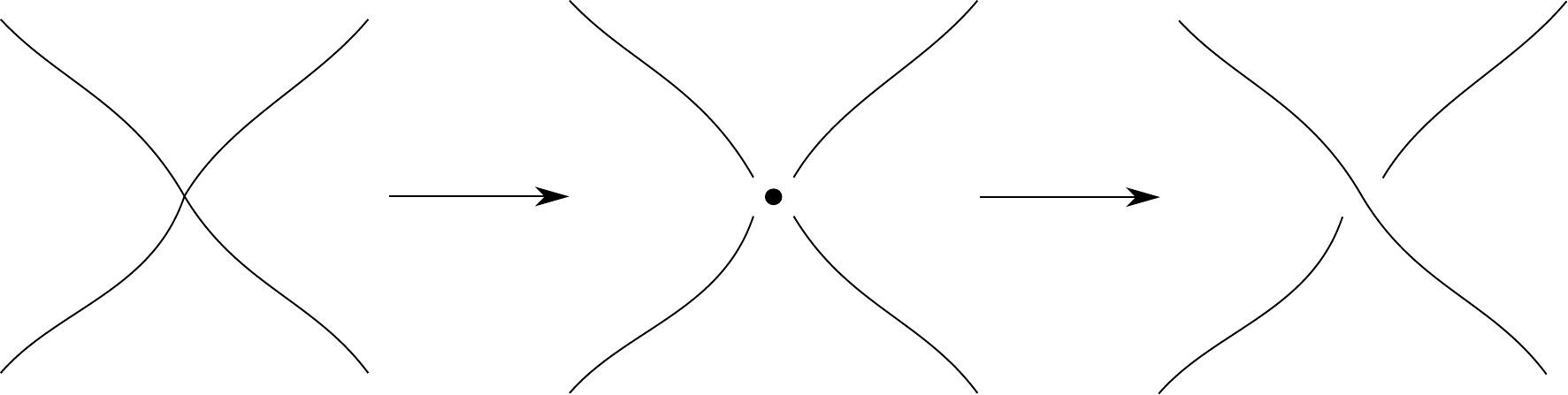}
 };
 \node at (-5.5,2.5){$\mathcal{C}$};
 \node at (-3,1.3){$\mathcal{C}\mapsto\tilde{\mathfrak X}$};
 \node at (-1.9,2.3){$\tilde{\mathcal{S}}_2$};
 \node at (-1.9,-.3){$\tilde{\mathcal{S}}_3$};
 \node at (.75,2.4){$\tilde{\mathcal{S}}_1$};
 \node at (.85,-.3){$\tilde{\mathcal{S}}_4$};
 \node at (.4,1){$\tilde{\mathcal{S}}_5$};
\node at (2.5,1.3){$\tilde{\mathfrak X}\mapsto\mathfrak X$};
\node at (6.25,2.3){${\mathcal{S}}_1$};
\node at (6.,0){${\mathcal{S}}_3$};
\node at (4.,0){${\mathcal{S}}_2$};
\end{tikzpicture}
    \caption[Stratifications of the critical variety]%
{Stratifications of the critical variety. On the left, the critical variety. In the centre, the Whitney stratification, $\Tilde{\mathfrak{X}}$, of $\mathcal{C}$. On the right, a possible stratification, $\mathfrak{X}$, of $C$ when the Whitney condition is relaxed.}
    \label{fig:strat_map}  
\end{figure}
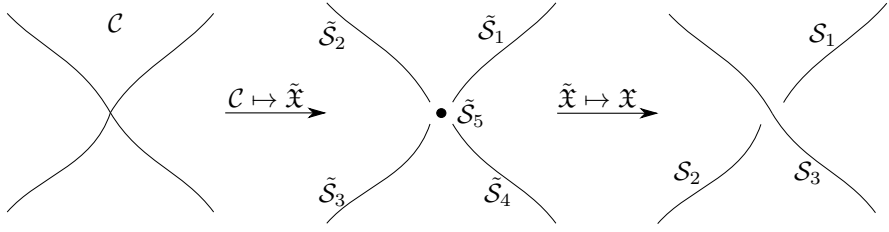

\begin{Proposition}\label{prop:singular_canards}
A singular canard exists if the vector field $\rho \circ X_1$ is a stratified vector field on the relaxed stratification $\mathfrak{X}$.   
\end{Proposition}
\begin{proof}
    Let us start with the Whitney stratification $\Tilde{\mathfrak{X}}$, and without loss of generality suppose we want to connect the two strata $\Tilde{\mathcal{S}}_2$, $\Tilde{\mathcal{S}}_4$. We recall that $\rho \circ X_1$ is already a stratified vector field on $\Tilde{\mathcal{S}}_2$, $\Tilde{\mathcal{S}}_4$. At the singular point, $p_s$, the map $\rho$ gives the vector $X_1(p_s)$. In order to connect with "nonzero velocity" the two strata $\Tilde{\mathcal{S}}_2$, $\Tilde{\mathcal{S}}_4$ it is necessary that the vector $X_1(p_s)$ is both tangent at the intersection of the closure of the two strata and compatible with the vector fields on the strata. In terms of stratifications, one considers the stratum $\mathcal{S}_3 := \Tilde{\mathcal{S}}_2 \cup \Tilde{\mathcal{S}}_4 \cup \Tilde{\mathcal{S}}_5$ belonging to a relaxed stratification $\mathfrak{X}$ and ask that $\rho \circ X_1$ induces a stratified vector field, in particular on $\mathcal{S}_3$. In this case, $\mathcal{S}_3$ is invariant under the (reduced) flow of $\rho\circ X_1$ making $\mathcal{S}_3$ a singular canard according to definition \ref{def:sing_canard}.
\end{proof}

\begin{Remark}
    To better understand the construction of proposition \ref{prop:singular_canards}, we show, in figure \ref{fig:connections}, a couple of cases where the singular canard connection is not present or not possible. The notation used in the proof of proposition \ref{prop:singular_canards} has been adapted to correspond to the one used in figures \ref{fig:strat_map} and \ref{fig:connections}.
\end{Remark}

\begin{figure}[ht]
\centering
  \begin{subfigure}[b]{0.45\textwidth}
         \centering
    \def\svgwidth{1\columnwidth}
\begingroup%
  \makeatletter%
  \providecommand\color[2][]{%
    \errmessage{(Inkscape) Color is used for the text in Inkscape, but the package 'color.sty' is not loaded}%
    \renewcommand\color[2][]{}%
  }%
  \providecommand\transparent[1]{%
    \errmessage{(Inkscape) Transparency is used (non-zero) for the text in Inkscape, but the package 'transparent.sty' is not loaded}%
    \renewcommand\transparent[1]{}%
  }%
  \providecommand\rotatebox[2]{#2}%
  \newcommand*\fsize{\dimexpr\f@size pt\relax}%
  \newcommand*\lineheight[1]{\fontsize{\fsize}{#1\fsize}\selectfont}%
  \ifx\svgwidth\undefined%
    \setlength{\unitlength}{141.73228346bp}%
    \ifx\svgscale\undefined%
      \relax%
    \else%
      \setlength{\unitlength}{\unitlength * \real{\svgscale}}%
    \fi%
  \else%
    \setlength{\unitlength}{\svgwidth}%
  \fi%
  \global\let\svgwidth\undefined%
  \global\let\svgscale\undefined%
  \makeatother%
  \begin{picture}(1,1)%
    \lineheight{1}%
    \setlength\tabcolsep{0pt}%
    \put(0,0){\includegraphics[width=\unitlength,page=1]{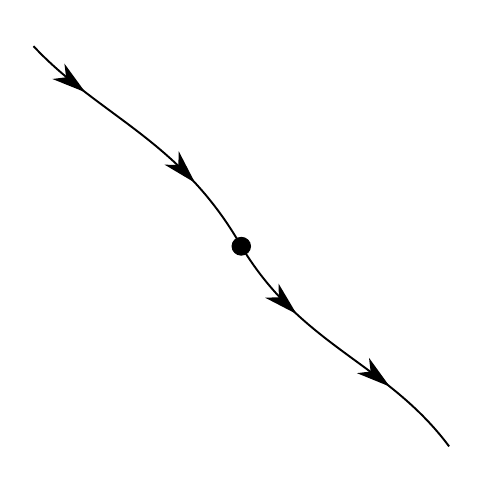}}%
    \put(0.26334734,0.78691439){\makebox(0,0)[lt]{\lineheight{1.25}\smash{\begin{tabular}[t]{l}$\Tilde{\mathcal{S}}_2$\end{tabular}}}}%
    \put(0.74590807,0.29405229){\makebox(0,0)[lt]{\lineheight{1.25}\smash{\begin{tabular}[t]{l}$\Tilde{\mathcal{S}}_4$\end{tabular}}}}%
    \put(0.54916143,0.48231484){\makebox(0,0)[lt]{\lineheight{1.25}\smash{\begin{tabular}[t]{l}$p_s$\end{tabular}}}}%
    \put(0,0){\includegraphics[width=\unitlength,page=2]{connection1.pdf}}%
  \end{picture}%
\endgroup%

         \caption{}
         \label{fig:connection1}
     \end{subfigure}
     \hfill
     \begin{subfigure}[b]{.45\textwidth}
         \centering
    \def\svgwidth{1\columnwidth}
\begingroup%
  \makeatletter%
  \providecommand\color[2][]{%
    \errmessage{(Inkscape) Color is used for the text in Inkscape, but the package 'color.sty' is not loaded}%
    \renewcommand\color[2][]{}%
  }%
  \providecommand\transparent[1]{%
    \errmessage{(Inkscape) Transparency is used (non-zero) for the text in Inkscape, but the package 'transparent.sty' is not loaded}%
    \renewcommand\transparent[1]{}%
  }%
  \providecommand\rotatebox[2]{#2}%
  \newcommand*\fsize{\dimexpr\f@size pt\relax}%
  \newcommand*\lineheight[1]{\fontsize{\fsize}{#1\fsize}\selectfont}%
  \ifx\svgwidth\undefined%
    \setlength{\unitlength}{141.73228346bp}%
    \ifx\svgscale\undefined%
      \relax%
    \else%
      \setlength{\unitlength}{\unitlength * \real{\svgscale}}%
    \fi%
  \else%
    \setlength{\unitlength}{\svgwidth}%
  \fi%
  \global\let\svgwidth\undefined%
  \global\let\svgscale\undefined%
  \makeatother%
  \begin{picture}(1,1)%
    \lineheight{1}%
    \setlength\tabcolsep{0pt}%
    \put(0,0){\includegraphics[width=\unitlength,page=1]{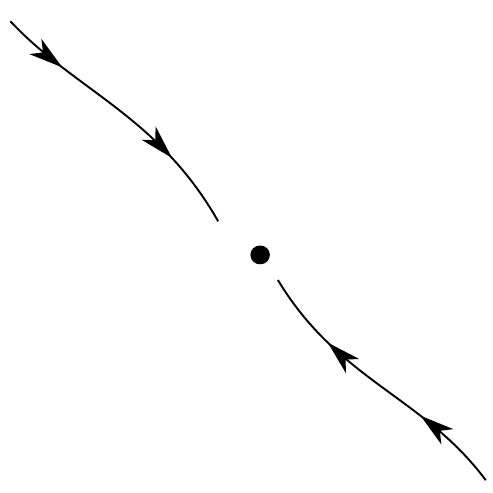}}%
    \put(0.31655722,0.78638572){\makebox(0,0)[lt]{\lineheight{1.25}\smash{\begin{tabular}[t]{l}$\Tilde{\mathcal{S}}_2$\end{tabular}}}}%
    \put(0.78420631,0.27648159){\makebox(0,0)[lt]{\lineheight{1.25}\smash{\begin{tabular}[t]{l}$\Tilde{\mathcal{S}}_4$\end{tabular}}}}%
    \put(0.5874596,0.46474421){\makebox(0,0)[lt]{\lineheight{1.25}\smash{\begin{tabular}[t]{l}$p_s$\end{tabular}}}}%
  \end{picture}%
\endgroup%

         \caption{}
         \label{fig:connection2}
     \end{subfigure}
        \caption{Figure \ref{fig:connection1} shows a case where the vector $X_1(p_s)$ fails to connect the two strata because it is not tangent to the curve. Indeed the vector field obtained via $\rho \circ X_1$ is not a vector field on the smooth curve considered. In figure \ref{fig:connection2} we can see a case where there exist no vector at $p_s$ inducing a smooth vector field, under the condition $\|X_1(p_s) \| =O(1)$.}
        \label{fig:connections}
\end{figure}

Essentially, proposition \ref{prop:singular_canards} states that a singular canard exists when the vector field $X_1$ at $p_s$ connects smoothly the stratified vector fields on the two strata we are connecting. Of course, such stratum corresponds to the solution of the reduced problem crossing the singular point with finite speed. By means of the stratification, the ambiguity arising at the intersection point is removed.

Let us recall that the critical variety $\mathcal{C}$ can be seen as the union of two smooth algebraic curves, $\mathcal{C}= V(F_1) \cup V(F_1)$, intersecting transversality at the singularity. Given the Whitney stratification $\Tilde{\mathfrak{X}} =\{\Tilde{\mathcal{S}}_1, \Tilde{\mathcal{S}}_2, \Tilde{\mathcal{S}}_3, \Tilde{\mathcal{S}}_4, \Tilde{\mathcal{S}}_5 \}$, there are two possible (relaxed) stratifications  $\mathfrak{X}$: $\{ V(F_1), \Tilde{\mathcal{S}}_2, \Tilde{\mathcal{S}}_4 \}, \{ V(F_1), \Tilde{\mathcal{S}}_1, \Tilde{\mathcal{S}}_3 \}$ (other combinations would not be smooth). So, from proposition \ref{prop:singular_canards} we deduce the following condition.

\begin{Lemma}\label{lm:condition}
If the vector field $X_1$ satisfies the condition $X_1(p_s) \wedge T_{p_s}V(F_1) = 0$, where $\wedge$ is the exterior product, then \eqref{eq:problem} admits singular canard solutions. Similarly for $V(F_2)$.
\end{Lemma}

\begin{Remark}
   In the standard form, conditions for singular canards in the case of transcritical and pitchfork singularities have been exploited in \cite{Krupa_2001}. A non-standard approach has been used in \cite{Maesschalck2015} to prove existence of canards for the unfolding of the vector field considered in \cite{Krupa_2001}. Our condition is in agreement with the results in the aforementioned manuscripts.
\end{Remark}

The analysis conducted in \cite{Krupa_2001, Maesschalck2015} not only provides conditions for singular canards but also proves their persistence under (appropriate) small perturbations. The persistence of canard solutions is a relevant result within the topic while relying on a detailed description of the system and rather quantitative techniques. In fact, proving the existence of canards for the unfolded problem usually relies on implicit function theorem arguments, while more quantitative estimates require, for example, an adaptation of Melnikov's method \cite{GuckenheimerHolmes}. Our approach is not quantitative in nature and we consider the problem from a wider perspective. For such reasons, there is no straightforward generalisation to the conditions of persistence (which we postpone to future works). Nevertheless, our method easily extends to degenerate singularities with pairwise transversal intersections, which we discuss in the following.

\subsection{Degeneracy}\label{sec:degeneracy}

In this section, we address some degenerate cases. Let us start by pointing out a desingularisation procedure. Let $F$ be the polynomial associated with the critical variety. In the previous section, we assumed a factorisation into irreducible components of the form $F=F_iF_j$. Since we assumed that the system has a singularity with an intersection number equal to one, then the corresponding critical variety is hyperbolic away from the singularity. However, if a term of the factorisation is of the form $(F_i)^n$, $n>1$, although the curve described is geometrically the same as the one given by $F_i$, the linearisation is not of full rank along the curve.

\begin{Lemma}
    Let $(F_i)^n$ be a term of the factorisation of $F$, with $n=2k+1$, $k\in \mathbb{N}$. The rescaling $F \mapsto F/(F_i)^{2k}$ desingularises the system.
\end{Lemma}

\begin{proof}
We want to show that near a section of the critical variety induced by $(F_i)^n$, with $n=2k+1$, away from the singularity, the system is topologically equivalent to the regularised system. Considering a small enough section, we can study the system given by
\begin{equation}\label{eq:simple_system}
    \begin{aligned}
    \dot x &= x^{2 k +1} + \epsilon \\
    \dot y &=\alpha \epsilon
    \end{aligned} ,
\end{equation}
where $\alpha$ is a real parameter. Equation \eqref{eq:simple_system} is obtained by considering the fact that the perturbation in a small enough neighbourhood of the critical variety can be considered constant. The section of the critical variety has been stretched to a straight line, and by a rotation the system is orthogonal to the fast-foliation. Let us point out that for completeness one should also study the cases where the equation for $\dot x$ is given by $x^{2 k +1} - \epsilon$ and $x^{2 k +1}$ respectively. However, there is no difference in the analysis we are going to show, for such reason these cases are omitted in the proof.

We perform a blow-up of system \eqref{eq:simple_system} along the $x$ and $\epsilon$ directions. The blow-up technique is briefly introduced in appendix \ref{app:blowup}. So, $\forall y$ the singular point $(x,\epsilon)=(0,0)$ become a circle. After desingularisation the motion on the circle is given by 
\begin{equation}\label{eq:motion_circle}
    \dot \psi = \frac{(2 k+1) \sin{\psi}  \left(\cos ^{2 k+1}(\psi )+\sin{\psi}\right)}{k \cos (2 \psi )-k-1} ,
\end{equation}
where $\psi \in [0, \pi ]$, i.e., $\psi$ is restricted to the values corresponding to positive $\epsilon$. The equilibria corresponding to the fast-foliation are given by $\psi=0$ and $\psi=\pi$, these equilibria are obtained by setting to zero the factor $\sin{\psi }$. We are left to study the equation $\cos ^{2 k}(\psi )+\tan{\psi}=0$. Let us notice that the function $\cos ^{2 k}(\psi )$ is positive semi-definite, and $\tan{\psi }$ is a monotone function in $] -\pi/2 , \pi/2[$, with image $\mathbb{R}$. Therefore, there is a unique solution in the interval $] -\pi/2 , \pi/2[$. By noticing that the equation $\cos ^{2 k}(\psi )+\tan{\psi}=0$ is periodic, with period $\pi$, we can conclude that there is a unique equilibrium for each interval with length $\pi$. So, we have that there is a unique equilibrium, $\psi^*$, in $[0, \pi ]$, $\forall k$.

By computing the derivative of \eqref{eq:motion_circle} we can study the stability properties. The equilibria $0,\pi$ are both stable and hyperbolic $\forall k$. As there is a unique equilibrium on the surface of the blown-up cylinder, we can further deduce that the equilibrium $\psi^*$ is a hyperbolic source. So, we can see that system \eqref{eq:simple_system} is topologically equivalent to the desingularised version, see figure \ref{fig:regularisation}. The case where the equilibria are sinks is analogous.

\begin{figure}[h]
\centering
    \resizebox{.3\linewidth}{!}{ 
    \def\svgwidth{1\columnwidth}
    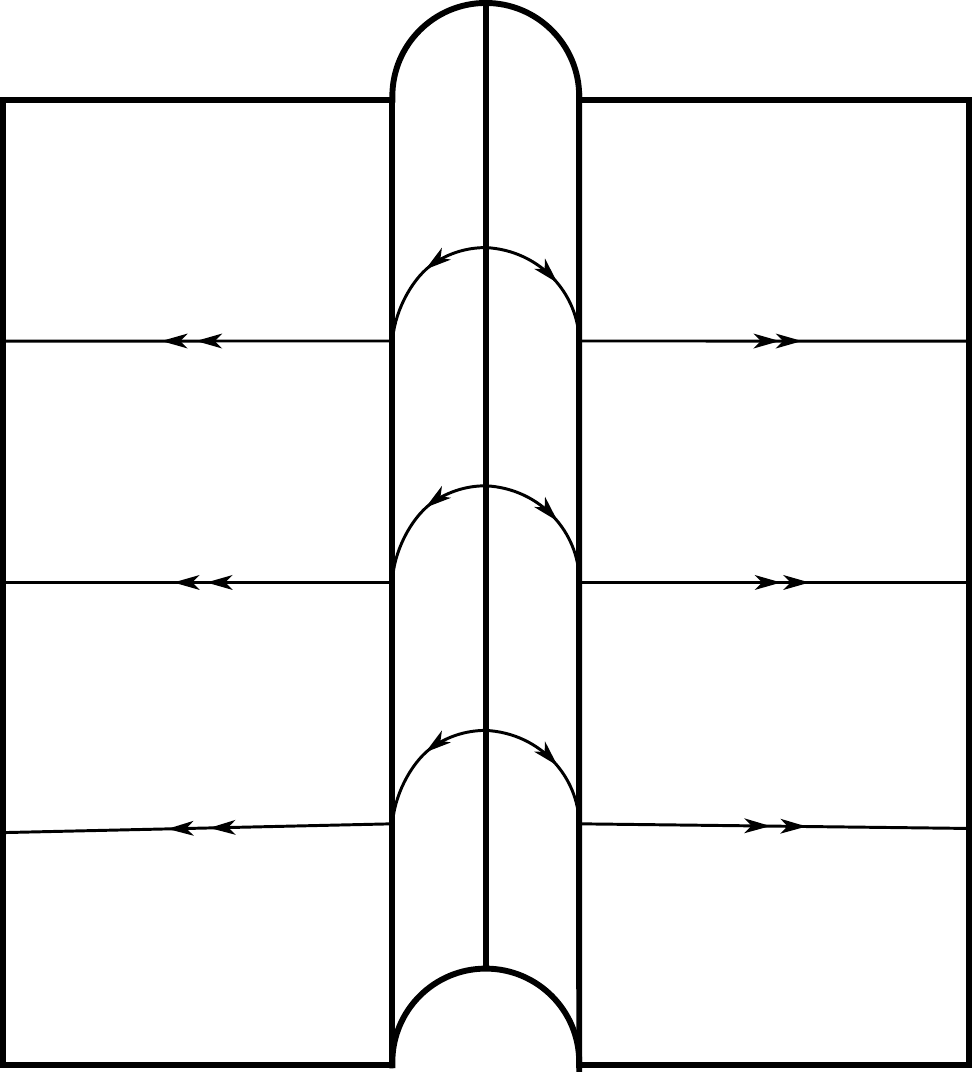
 
       }
       \caption{A section of $V(F_i^{2k+1})$ after the blow-up. We show the fast-foliation together with the motion on the cylindrical blown-up space.}
       \label{fig:regularisation}
\end{figure}
\end{proof}

Briefly speaking, terms of the form $(F_i)^n$ with odd powers can be easily desingularised and then our results apply. In contrast, even powers lead to a more problematic situation, that is not desingularisable as in the odd case. Heuristically, we could think of the equation $x^2 + a$, where $a$ is a real constant. Depending on the value of $a$ we can have $0$, $1$, or $2$ solutions, where the unique solution is obtained only for $a=0$. So, we deduce that, in general, sections of the critical variety corresponding to terms of the form $(F_i)^{2k}$ do not persist, but either they disappear, or double. So, even powers require a dedicated analysis that we do not perform here.

Now, let us consider the case where the polynomial $F$ admits the factorisation $F_1 \dots F_N$, with $N\geq2$, such that $V(F_1)  \cup \dots \cup V(F_N) = \mathcal{C}$, $V(F_1)\cap \dots \cap V(F_N) = p_s$.

\begin{Definition}\label{def:transversal}
    Given $N$ algebraic curves intersecting at $p_s$, we say that the intersection is pairwise transversal if each pair of curves, $V(F_i), V(F_j)$ has intersection number one at $p_s$, $\forall i,j \in \{1, \dots, N \}$.
\end{Definition}

Pairwise transversality ensures that there are no curves of equilibria having the same tangent at $p_s$. So, under the setting of definition \ref{def:transversal}, we can repeat the procedure done for the case of two curves. The potential singular canards now are $N$, namely $V(F_1), \dots, V(F_N)$. Condition \ref{lm:condition} will read as follows.

\begin{Proposition}\label{prop:mult}
If the vector field $X_1$ satisfies the condition $X_1(p_s) \wedge T_{p_s}V(F_i) = 0$, then the system admits a singular canard along $V(F_i)$, where $i=1, \dots, N$.
\end{Proposition}

\subsection{Examples}\label{sec:examples}

As we already mentioned in the previous sections, the generic cases of transcritical and pitchfork singularities have been widely studied. As a consequence, a check of agreement between our results and the known conditions for singular canards in the non-degenerate setting is straightforward. Hereby we present two examples of degenerate singularities: a degenerate transcritical and a degenerate pitchfork, both satisfying the conditions of definition \ref{def:transversal}.
We are going to verify that, given the conditions of proposition \ref{prop:mult}, we can actually connect an attracting branch of the critical variety with a repelling one. To verify our claims, we employ a known technique: the blow-up, see appendix \ref{app:blowup}.

Let us consider the unperturbed vector field
\begin{equation}\label{eq:x0_transcritical}
    X_0 = \left( (y-x)(y+x)(y-x/2)(y+x/2) , 0 \right)
\end{equation}
The point $p_s=(0,0)$ is at the intersection of four lines of equilibria. Suppose we are interested in the canard following the branch $y=x/2$, which is attracting for $x$ negative and repelling for $x$ positive. Then, following the prescription of proposition \ref{prop:mult} we consider the perturbation term
\begin{equation}\label{eq:x1_transcritical}
    X_1 = \left( 1, 1/2 \right) .
\end{equation}
So, the vector field we study is given by $X=X_0 + \epsilon X_1$, see figure \ref{fig:transcritical}. By using a spherical blow-up on the singular point $(0,0)$, after desingularisation it is possible to identify the desired connection between the attracting and the repelling branch of the critical variety. Such a connection appears as a geodesic solution on the blow-up sphere, see figure \ref{fig:mult_transcritical}.

\begin{Remark}
    In this example, the existence of the singular canard could be also proven by noticing that the curve $y=x/2$ is invariant under the flow of $X$. Considering higher order terms, e.g., $\epsilon^2 X_2 + \epsilon^3 X_3 + \dots$ the line $y=x/2$ would not be invariant anymore. However, our claims and the blow-up picture still apply.
\end{Remark}

For the second example, we consider a degenerate pitchfork singularity given by the unperturbed vector field
\begin{equation}\label{eq:x0_pitchfork}
    X_0 = \left( (x +  y/2) (x -  y/2) (y - x^2) , 0 \right).
\end{equation}
The critical variety is given by a parabola and two straight lines intersecting at the origin. This case is interesting because, unlike the nondegenerate pitchfork, the parabola changes stability. So, we look for the canard following the parabola transitioning from the attracting to the repelling section. Since the tangent to the parabola at the origin is parallel to the $x$ axis, a simple choice for the perturbation satisfying proposition \ref{prop:mult} is given by
\begin{equation}\label{eq:x1_pitchfork}
    X_1 = - \left( 1, x \right) ,
\end{equation}
where the minus sign gives the correct direction, and the $y$ component provides a compatible flow on the normally hyperbolic sections without introducing equilibria of the reduced flow.
Again, by studying the vector field $X=X_0+\epsilon X_1$, see figure \ref{fig:pitchfork}, via a spherical blow-up we can confirm the presence of a canard connection between the attracting and repelling branches of the critical parabola, see figure \ref{fig:mult_pitchfork}.

For more details on the blow-up computations of both examples we refer to section \ref{sec:blowup_examples} of the appendix.

\begin{figure}
     \begin{center}
         \begin{subfigure}[b]{0.4\textwidth}
         \centering
    \def\svgwidth{1\columnwidth}
    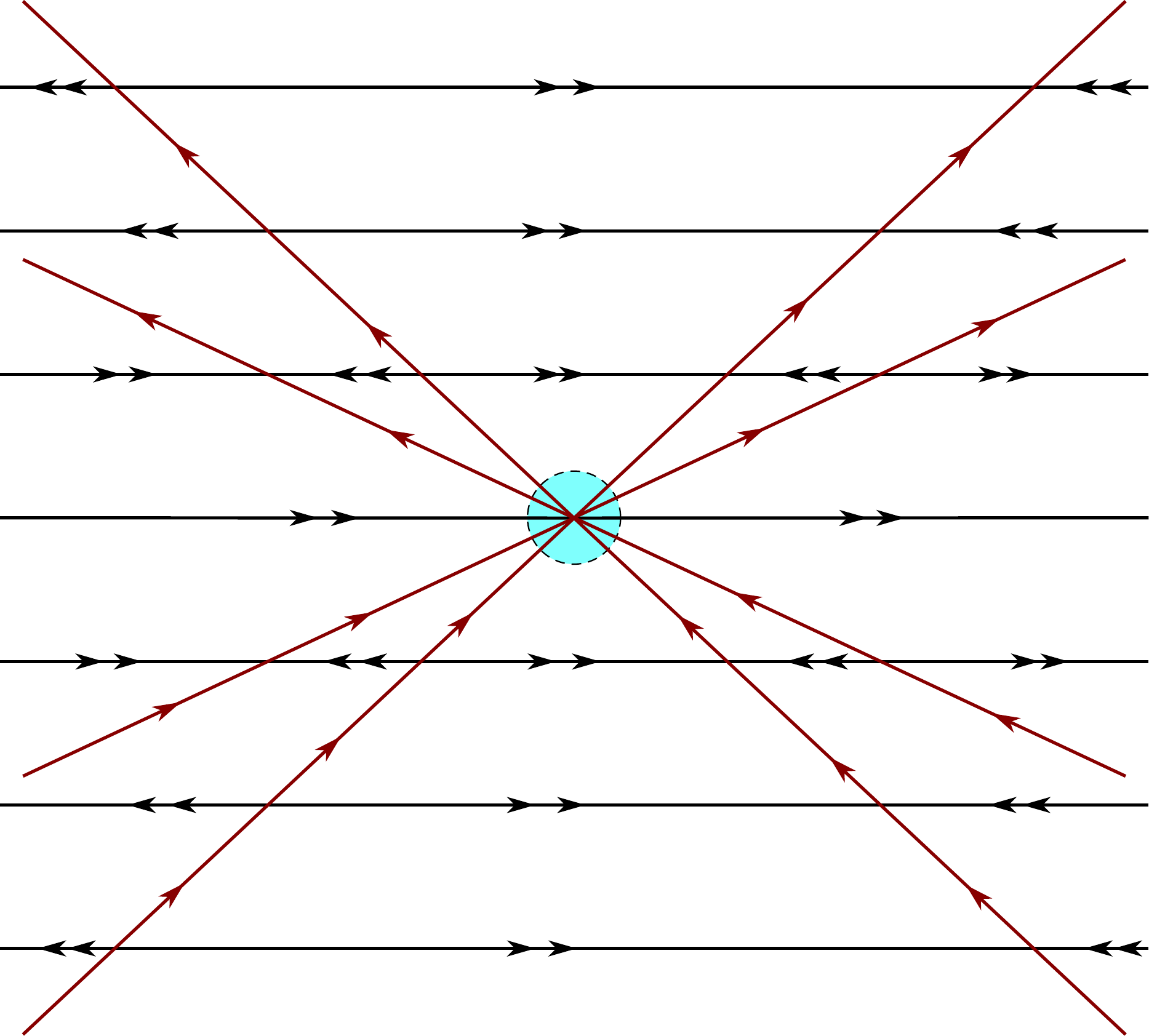

         \caption{Degenerate transcritical}
         \label{fig:transcritical}
     \end{subfigure}
     \hfill
      \begin{subfigure}[b]{0.4\textwidth}
         \centering
    \def\svgwidth{1\columnwidth}
    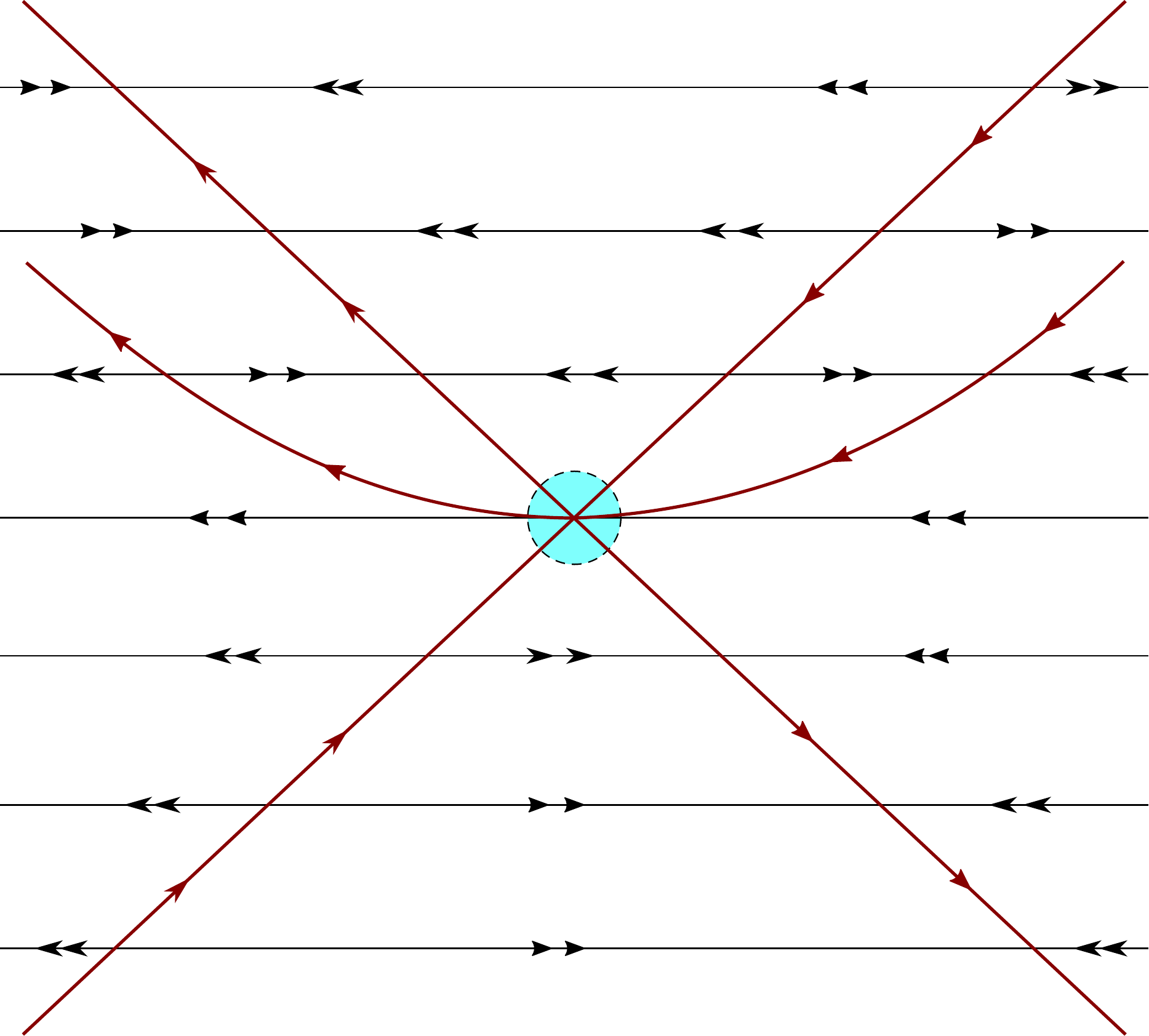

         \caption{Degenerate pitchfork}
         \label{fig:pitchfork}
     \end{subfigure}
     \end{center}
     \vfill
     \begin{center}
         \begin{subfigure}[b]{0.45\textwidth}
         \centering
    \def\svgwidth{1\columnwidth}
    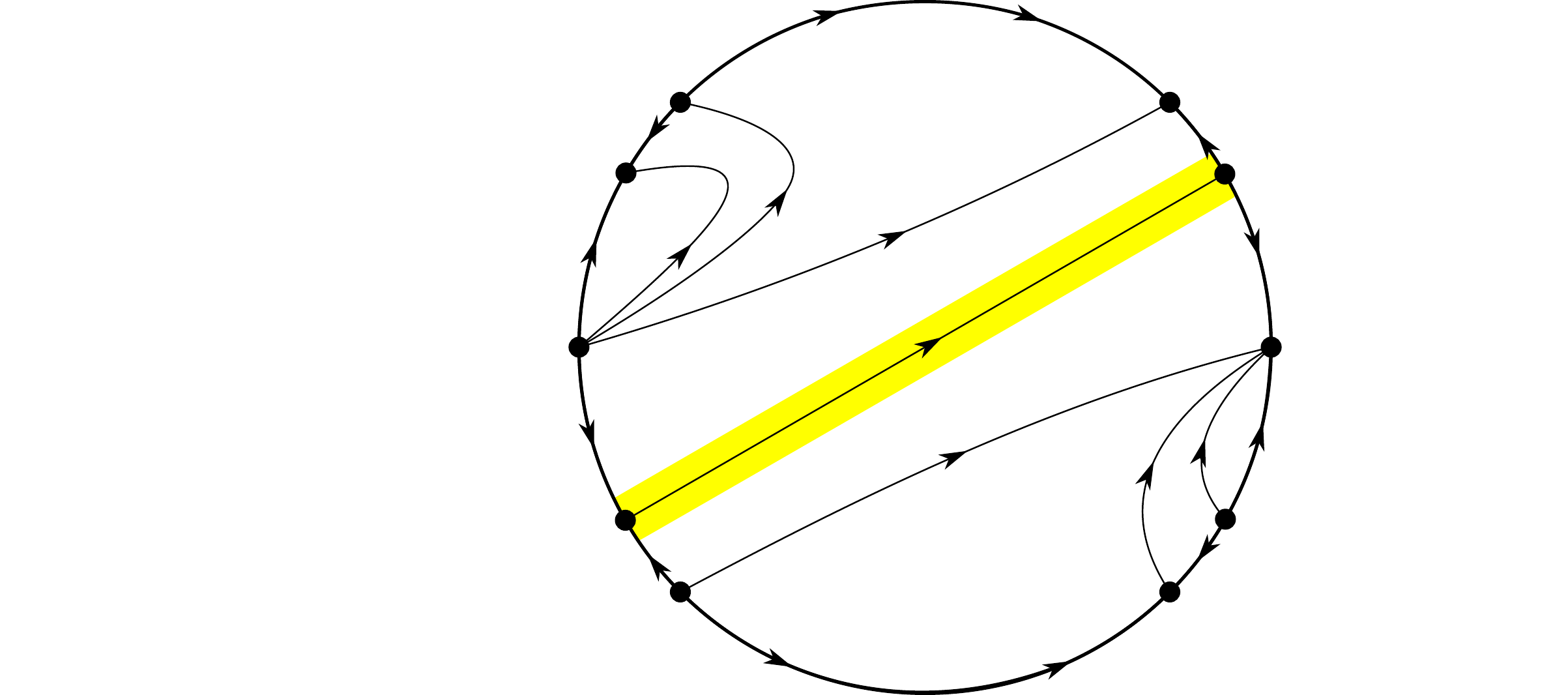

         \caption{Degenerate transcritical - Blow-up}
         \label{fig:mult_transcritical}
     \end{subfigure}
     \hfill
     \begin{subfigure}[b]{0.45\textwidth}
         \centering
    \def\svgwidth{1\columnwidth}
\begingroup%
  \makeatletter%
  \providecommand\color[2][]{%
    \errmessage{(Inkscape) Color is used for the text in Inkscape, but the package 'color.sty' is not loaded}%
    \renewcommand\color[2][]{}%
  }%
  \providecommand\transparent[1]{%
    \errmessage{(Inkscape) Transparency is used (non-zero) for the text in Inkscape, but the package 'transparent.sty' is not loaded}%
    \renewcommand\transparent[1]{}%
  }%
  \providecommand\rotatebox[2]{#2}%
  \newcommand*\fsize{\dimexpr\f@size pt\relax}%
  \newcommand*\lineheight[1]{\fontsize{\fsize}{#1\fsize}\selectfont}%
  \ifx\svgwidth\undefined%
    \setlength{\unitlength}{657.11755724bp}%
    \ifx\svgscale\undefined%
      \relax%
    \else%
      \setlength{\unitlength}{\unitlength * \real{\svgscale}}%
    \fi%
  \else%
    \setlength{\unitlength}{\svgwidth}%
  \fi%
  \global\let\svgwidth\undefined%
  \global\let\svgscale\undefined%
  \makeatother%
  \begin{picture}(1,0.59307084)%
    \lineheight{1}%
    \setlength\tabcolsep{0pt}%
    \put(0,0){\includegraphics[width=\unitlength,page=1]{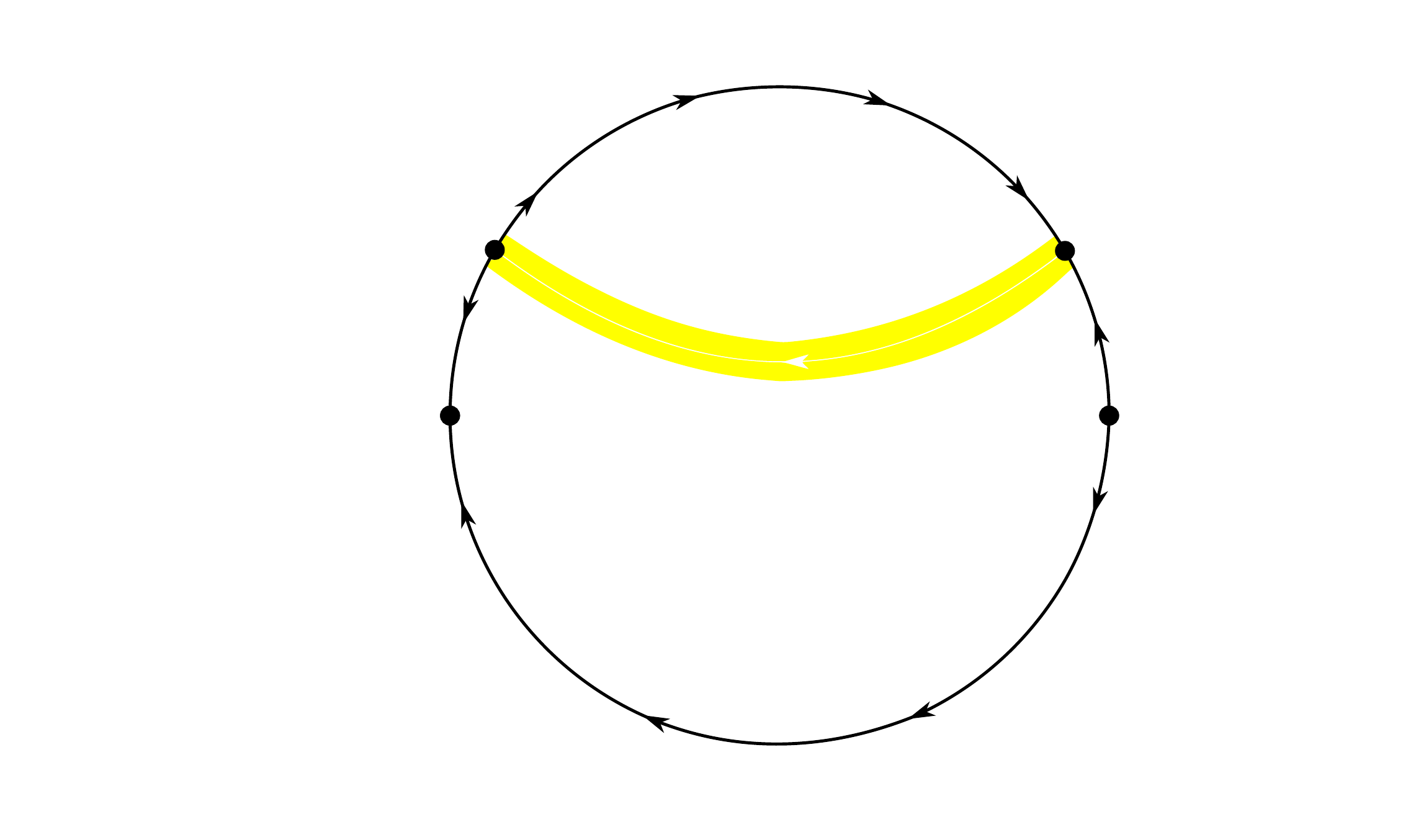}}%
    \put(0.79734652,0.29505391){\makebox(0,0)[lt]{\lineheight{1.25}\smash{\begin{tabular}[t]{l}$0$\end{tabular}}}}%
    \put(0,0){\includegraphics[width=\unitlength,page=2]{mult_pitchfork.pdf}}%
    \put(0.76795542,0.4193416){\makebox(0,0)[lt]{\lineheight{1.25}\smash{\begin{tabular}[t]{l}$\arctan{\frac{1}{\varphi}}$\end{tabular}}}}%
    \put(0.52688089,0.58136012){\makebox(0,0)[lt]{\lineheight{1.25}\smash{\begin{tabular}[t]{l}$\frac{\pi}{2}$\end{tabular}}}}%
    \put(-0.00115175,0.43893055){\makebox(0,0)[lt]{\lineheight{1.25}\smash{\begin{tabular}[t]{l}$\pi-\arctan{\frac{1}{\varphi}}$\end{tabular}}}}%
    \put(0.27034901,0.29691462){\makebox(0,0)[lt]{\lineheight{1.25}\smash{\begin{tabular}[t]{l}$\pi$\end{tabular}}}}%
    \put(0.48708879,0.00306143){\makebox(0,0)[lt]{\lineheight{1.25}\smash{\begin{tabular}[t]{l}$-\frac{\pi}{2}$\end{tabular}}}}%
    \put(0,0){\includegraphics[width=\unitlength,page=3]{mult_pitchfork.pdf}}%
  \end{picture}%
\endgroup%

         \caption{Degenerate pitchfork - Blow-up}
         \label{fig:mult_pitchfork}
     \end{subfigure}
     \end{center}
        \caption{
        The upper pictures (\ref{fig:transcritical}, \ref{fig:pitchfork}) are representative phase portraits where the fast-foliation (double-arrow black lines) and the reduced flow (single-arrow red lines) are shown. A neighbourhood of the singularity has been highlighted by a light blue disk. Figure \ref{fig:transcritical} is associated with the vector field $X_0 + \epsilon X_1$ given by \eqref{eq:x0_transcritical}, \eqref{eq:x1_transcritical}. While figure \ref{fig:pitchfork} is associated with the vector field $X_0 + \epsilon X_1$ given by \eqref{eq:x0_pitchfork}, \eqref{eq:x1_pitchfork}. The pictures below (\ref{fig:mult_transcritical}, \ref{fig:mult_pitchfork}) represent the blow-up of the singularities. In particular, we are showing the hemispheres corresponding to positive values of $\epsilon$. Along the (invariant) equatorial line the equilibria emerging from the desingularisation process are displayed together with the polar coordinate. The equilibria at $0$ and $\pi$ correspond with the fast-foliation, while the other equilibria can be traced back to the different branches of the critical variety. In figure \ref{fig:mult_transcritical}, we can see the blow-up of the degenerate transcritical singularity for the system displayed in fig.\ref{fig:transcritical}. The canard is given by the highlighted solution connecting the attracting equilibrium at $\pi + \arctan{1/2}$ to the repelling equilibrium at $\arctan{1/2}$. In figure \ref{fig:mult_pitchfork}, we can see the blow-up of the degenerate pitchfork singularity for the system displayed in in fig.\ref{fig:pitchfork}. Let us notice that in this case one spherical blow-up it is not enough to fully desingularise the system, indeed for a detailed local behaviour near the points $\pm \pi/2$ further blow-ups are necessary. At any rate, the blow-up we performed is enough to establish the connection (highlighted curve) between the attracting and repelling branches of the parabola, respectively at $\arctan{1/\varphi}$ and $\pi -\arctan{1/\varphi}$, where $\varphi$ is the golden ratio. 
        }
        \label{fig:degenerate_cases}
\end{figure}

\section{Discussion and perspectives}

In this paper, we studied the presence of singular canards for polynomial vector fields of the from $X=X_0 + \epsilon X_1$, when a singularity due to the intersection of several branches of the critical variety is present. Let us remark that our results could also be applied to smooth (not necessarily polynomial) vector fields by taking the zeroth and first \emph{jets} in $\epsilon$ and restricting the domain to a neighbourhood of the singularity. Also, one can extend the results to (smooth) vector fields on Riemannian manifolds, if there exists a local chart covering a neighbourhood of the singularity.

Let us recall that canard solutions act as a separatrix between generic behaviours. Therefore, knowing the conditions for canards allows us also to predict the behaviour of the system under a given perturbation.

Singularities arising from tangency with the fast-foliation, such as in the case of the fold, have been omitted. In principle, in the case of fold singularities the concept of stratification is not necessary, since there is only one critical curve. By comparison with the results in the literature \cite{Szmolyan2001}, one can check that the tangency condition we found, see lemma \ref{lm:condition}, works also in the case of the fold. Indeed, one would ask the perturbation to be tangent to the parabola at the fold point. However, some new aspects arise. Notice that, generically, in an $\epsilon$-neighbourhood of the fold, the vector field $X_0$ is of order $\epsilon$, since only one curve of equilibria passes through the singular point. Such a property allows coordinate transformations to absorb part of the perturbation terms, leading to potentially more general conditions on the perturbation. For such a reason, we omitted the treatment here and postpone the problem to future works. As a further development, we can consider those degenerate cases where the intersection number between two curves is greater than one. For example, when the intersection of two curves is not transversal and so the two curves share the tangent space at the intersection. Unfolding the degenerate problems we have discussed can also be interesting.

We conclude by briefly digressing on slow-fast discrete-time planar maps.

\subsection{Discrete planar maps}
\newcommand{\R}{\mathbb R}

The theory of slow-fast maps is considerably limited compared to the continuous-time counterpart. Nevertheless, in this section, we argue that the theory developed in the main part of this paper holds as well for singularly perturbed discrete-time maps. Particularly relevant for the coming digression are \cite{arcidiacono2019discretized,engel2019discretized}, which deal with pitchfork and transcritical singularities, respectively. For a recent development towards Discrete GSPT see \cite{jelbart2023discrete}. Moreover, \cite[Appendix A]{szmolyan2004relaxation} is particularly useful within the context of normal hyperbolicity of $2$-dimensional slow-fast maps.

Given the planar vector field \eqref{eq:problem}, we can consider the discrete map obtained via Euler discretisation, namely a map $P:q\mapsto\tilde q$ defined by
\begin{equation}
    \tilde q = q +  \left( X_0  + \epsilon  X_1 \right) \delta ,
\end{equation}
where $q \in \mathbb{R}^2$,  $\delta$ is the discretisation step, and abusing notation we re-use the notation $X_i$ also for the discrete maps. 

In this setting, the critical set is $ \mathcal C=\left\{ p\in\R^2\, |\, X_0(p)=0\right\}$, which geometrically coincides with its continuous-time counter part. In what follows, just as in the main text, we assume that $p_s\in\mathcal C$ is an isolated singular point where the branches of the critical set have pairwise intersection number one. The stratification arguments hold the same with the appropriate adaptation for the Whitney regularity condition, that is whenever the linear map $\text{D}P|_{\mathcal C, \epsilon=0}=\text{Id}+\delta\text{D}X_0$ has no multipliers on the unit circle. Naturally, for $\delta\neq0$ this coincides with the continuous-time regularity condition. In other words, the Whitney stratification of $\mathcal C$ is the same for the continuous and for the discrete-time problems. We, therefore, use the same notation $\tilde{\mathfrak X}$ and $\mathfrak X$ to denote the Whitney and the relaxed stratifications that we introduced in section \ref{sec:strat}.

Regarding geometric properties, one can verify that in the setting of \cite{jelbart2023discrete}, the one-dimensional strata of $\mathcal C$ are, as in the continuous-time case, normally hyperbolic. As shown in \cite{jelbart2023discrete}, many of the concepts of the continuous-time setting have discrete-time analogues. In particular, one keeps important objects such as the \emph{fast-foliation}, the \emph{layer map}, and the \emph{reduced map}. What is most relevant for our discussion, is that the reduced map in the discrete-time setting has the exact same geometric interpretation as for the continuous-time one.

Regarding singular canards, and completely analogous to definition \ref{def:sing_canard}, we shall say that, for a slow-fast discrete-time map, a singular canard is an orbit, of iterations of the reduced map, passing through the singular point with nonzero finite speed. Moreover, in turn, analogous to definition \ref{def:stratified_vf}, we shall say that a stratified map is a map that assigns to each stratum a smooth map. If we use the exact same notation $\rho$ for the projection along the fast-foliation, we see that the reduced map defined by $\rho\circ X_1$ is a stratified map.

So, we can now argue that as in proposition \ref{prop:singular_canards}, for discrete-time slow-fast maps a singular canard exists if $\rho\circ X_1$ is a stratified map on the (relaxed) stratification $\mathfrak X$, see also lemma \ref{lm:condition}.

\vspace{6pt} 



\appendix
\section{Projections and the reduced problem}\label{app:projection}

Let us recall here some key concept from GSPT \cite{wechselberger2020geometric}. The vector field $X_0$ defines a foliation of the phase-space that generically is transverse to the critical variety $\mathcal{C}$, this is the so-called \emph{fast-foliation}. The reduced problem, or reduced flow, of a singularly perturbed system is defined via the unique projection map, $\pi : T_{\mathcal{C}}\mathbb{R}^2 \to T\mathcal{C}$, which takes vectors in $\mathbb{R}^2$ with base point in $\mathcal{C}$ and projects them onto the tangent space of the critical variety along the fast-foliation generated by $X_0$. So, the reduced problem is defined by the vector field $\pi \circ X_1$ on (regular sections of) $\mathcal{C}$. As immediate consequence of the definition of the projection, we have that such map is not well-defined when the fast-foliation is tangent to the critical variety. Such a property can be used to define, for example, \emph{fold points and folded singularities}.

Let us notice that, for the class of problems under our consideration, the singularities arise from the geometrical properties of the critical variety, i.e., there is a point of (self-)intersection. So, in such cases, and in contrast to for example fold points, the projection is not defined because there is no well-defined tangent space of the critical variety at the intersection point. Moreover, from the fact that Whitney strata have maximal rank, we can exclude the possibility of a whole branch of the critical manifold being aligned with the fast-foliation.

\section{Blow-Up}\label{app:blowup}

Given a nilpotent equilibrium point of a vector field, the blow-up is a desingularisation process that transforms the singular point in a higher-dimensional manifold with the aim to retrieve `more' hyperbolicity. Such a technique has been widely studied and applied with success to many problems concerning singular perturbations \cite{Kojakhmetov2021, kuehn2016multiple, Ferragut2011, Krupa_2001, Maesschalck2015, SZMOLYAN2001419, Szmolyan2001,dumortier1996canard}.

Let us briefly describe the \emph{spherical blow-up} for the setting of this paper. First, we consider $\epsilon$ not as a parameter, but as a coordinate. So, our two-dimensional vector field, $X$, now becomes a three-dimensional vector field $\hat{X}$. The singular point is given by the coordinate of $p_s$ together with $\epsilon=0$, we call this nilpotent point $\hat{p}_s$. The spherical blow-up is given by the transformation $\phi : \mathbb{R}^{2+1} \to \mathbb{S}^2_{\hat{p}_s} \times \mathbb{R}_+$, where $\mathbb{S}^2_{\hat{p}_s}$ is the two-sphere centred at $\hat{p}_s$. The restriction of the blown-up system to $\mathbb{S}^2_{\hat{p}_s} \times \{0\}$ corresponds to the singular point. However, the transformation $\phi$ by itself is not enough to obtain more hyperbolicity on the sphere $\mathbb{S}^2_{\hat{p}_s} \times \{0\}$. In order to \emph{desingularise} the system it is necessary to perform a \emph{conformal transformation} such that the system on $\mathbb{S}^2_{\hat{p}_s} \times \{0\}$ is non-trivial, and in the best scenario fully hyperbolic.

\subsection{Blow-ups of examples of section \ref{sec:examples}}\label{sec:blowup_examples}

The first example leads to the equations
\begin{equation}
    \begin{aligned}
\dot x &= (y-x)(y+x)(y-x/2)(y+x/2) + \epsilon \\
\dot y &=   \epsilon/2 \\
\dot \epsilon &= 0
    \end{aligned}.
\end{equation}
The spherical blow-up is given by the following transformation: $x\to r \cos{\theta} \sin{\phi}$, $y\to r \sin{\theta } \sin{\phi}$, $\epsilon \to r^4 \cos{\phi}$, where the powers of $r$ have been chosen calibrating the powers in the $\dot x$ equation. The appropriate desingularisation is obtained dividing the blown-up vector field by $r^3$. At this point, by setting $r$ to zero we obtain the desingularised vector field on the blown-up sphere,
\begin{equation}\label{eq:bu_transcritical}
    \begin{aligned}
       &\dot \theta = \frac{1}{16} \Big[\sin \theta  (6 \cos (2 \theta )-5 (\cos (4 \theta )+1)) \sin ^3(\phi )+8 \cot \phi  (\cos \theta -2 \sin \theta )\Big] \\
       &\begin{split}
         \dot \phi =  \frac{\cos \phi}{6 \cos (2 \phi )+10}   \Big[\cos \theta  (-6 \cos (2 \theta )+5 \cos (4 \theta )+5) \sin ^4(\phi ) \\
        +8 \cos \phi  (\sin \theta +2 \cos \theta )\Big]   
       \end{split}        
    \end{aligned} .
\end{equation}
Setting $\phi=\pi/2$ in \eqref{eq:bu_transcritical}, one can check that the equatorial line is invariant, and the motion on the equator is given by
\begin{equation} \label{eq:equator_transcritical}
   \dot \theta = -\frac{1}{16} \sin \theta  (-6 \cos (2 \theta )+5 \cos (4 \theta )+5) .
\end{equation}
The equilibria of \eqref{eq:equator_transcritical} are $\theta=0, \pi,  \pm \pi/4 ,\pm 3/4 \pi, \pm \arctan{1/2}, \pm (\pi - \arctan{1/2})$. By setting $\theta=-\pi +\arctan{1/2}$ in \eqref{eq:bu_transcritical}, we obtain the connection between the points $\pm (\pi - \arctan{1/2})$ corresponding to the canard.

The second example is given by 
\begin{equation}
    \begin{aligned}
\dot x &= (x +  y/2) (x -  y/2) (y - x^2) + \epsilon \\
\dot y &=   \epsilon x \\
\dot \epsilon &= 0
    \end{aligned} .
\end{equation}
In this case, we choose a spherical blow-up where the powers of $r$ are calibrated to compensate the parabola's equation, i.e., $x\to r \cos{\theta} \sin{\phi}$, $y\to r^2 \sin{\theta} \sin{\phi}$, $\epsilon \to r^4 \cos{\phi}$. The desingularisation is given once again dividing by $r^3$. So, the vector field on the sphere at $r=0$ reads
\begin{equation}\label{eq:bu_pitchfork}
\begin{aligned}
    &\begin{split}
        \dot \theta = -&\frac{1}{2 \left(\cos (2 \theta )-8 \csc ^2(\phi )+5\right)}
            \Big[2 \sin{\theta} \cos ^4(\theta ) (5 \sin{\phi }+\sin (3 \phi )) \\
            &\hspace{2cm} -4 \cos ^2(\theta ) \left(\sin ^2(\theta ) \cos (2 \phi )+\cos{\phi}+4 \cos{\phi} \cot ^2(\phi )\right) \\
            &\hspace{2cm} +4 \sin{\theta} (\cos (2 \phi )+3) \cot{\phi} \csc ^2(\phi )-3 \sin ^2(2 \theta )  \Big]
    \end{split}\\
     &\begin{split}
         \dot \phi &=-\frac{16 \cos{\theta} \cos{\phi}}{-2 \cos (2 \theta ) \sin ^2(\phi )+5 \cos (2 \phi )+11}\Big[ \cos ^4(\theta ) \sin ^4(\phi )\\
         &\hspace{2cm} -\sin{\theta} \cos ^2(\theta ) \sin ^3(\phi )+\sin{\theta} \sin{\phi} \cos{\phi}+\cos{\phi} \Big]
     \end{split}
\end{aligned}
\end{equation}
Of course, the equator is invariant and the vector field on it is 
\begin{equation}\label{eq:equatorial_pitchfork}
   \dot \theta= -\frac{-3 \sin ^2(2 \theta )+8 \sin{\theta} \cos ^4(\theta )+4 \sin ^2(\theta ) \cos ^2(\theta )}{2 (\cos (2 \theta )-3)}.
\end{equation}
The equilibria of \eqref{eq:equatorial_pitchfork} are given by $\theta=0,\, \pi,\, \pm \pi/2,\, \arctan{1/\varphi},\, \pi -\arctan{1/\varphi}$, where $\varphi$ is the golden ratio.

For this example, it is not possible to obtain an analytical expression for the canard. However, one can check that \eqref{eq:bu_pitchfork} has a symmetry that enforces the presence of the canard. Let us consider the reflection with respect to the meridian passing trough $\theta=\pi/2$. In terms of coordinates, we first translate the system so that the axis $\theta=\pi/2$ is at $\theta=0$ and then we apply the reflection $R$, such that $R \theta \to -\theta$. Let $(\dot \theta, \dot \phi)$ be the translated system, then the action of the reflection on the vector field gives $R ( \dot \theta, \dot \phi) = ( \dot \theta, - \dot \phi)$. Such a symmetry implies that integral curves are symmetric with respect to reflection along the meridian passing through $\theta=\pi/2$. Consequently, the canard connection is established.


\bibliography{bibliography.bib}
\bibliographystyle{siam}

\end{document}